\documentclass{article}

\title{\bf Explicit real-part estimates for high order derivatives of analytic functions}

\author{\sc{Gershon Kresin} \\ \\
Department of Computer Science and Mathematics,\\ 
Ariel University, Ariel 40700, Israel\\
e-mail: kresin@ariel.ac.il}

{ \date\ }

\usepackage{amsthm,amsmath,amssymb,amsbsy,amsfonts}

\numberwithin{equation}{section}
\newtheorem{lemma}{Lemma}
\newtheorem{theorem}{Theorem}
\newtheorem{proposition}[theorem]{Proposition}
\newtheorem{corollary}{Corollary}

\begin{document}
\maketitle

\vspace{-10mm}
{\bf Abstract.}
The representation for the sharp constant ${\rm K}_{n, p}$ in an estimate of the modulus
of the $n$-th derivative of an analytic function in the upper half-plane ${\mathbb C}_+$ is considered. 
It is assumed that the boundary value of the real part of the function on 
$\partial{\mathbb C}_+$ belongs to  $L^p$. The representation for ${\rm K}_{n, p}$ comprises an optimization 
problem by parameter inside of the integral. This problem is solved for $p=2(m+1)/(2m+1-n)$,
$n\leq 2m+1$, and for some first derivatives of even order in the case $p=\infty$.
The formula for ${\rm K}_{n,\; 2(m+1)/(2m+1-n)}$ contains, for instance, the known expressions 
for ${\rm K}_{2m+1, \infty}$ and ${\rm K}_{m, 2}$ as particular cases.
Also, a two-sided estimate for ${\rm K}_{2m, \infty}$ is derived, which leads to the 
asymptotic formula ${\rm K}_{2m, \infty}=2\big ((2m-1)!!\big )^2/\pi + O\big ( \big ((2m-1)!!\big )^2 /(2m-1)\big )$ 
as $m \rightarrow \infty $. The lower and upper bounds of ${\rm K}_{2m, \infty}$ are compared with its value 
for the cases $m=1, 2, 3, 4$. As applications, some real-part theorems with explicit constants
for high order derivatives of analytic functions in subdomains of complex plane are described. 


\setcounter{equation}{-1}
\setcounter{section}{-1}
\section{Introduction}

In this paper we deal substantially with the coefficient $K_{n,p}(\alpha )$ in the inequality 
\begin{equation} \label{Eq_0.2}
|\Re \{ e^{i \alpha}f^{(n)}(z)\}| \leq {K_{n,p}(\alpha ) \over (\Im z)^{n+{1\over p}}}\;||\Re f||_p\;,
\end{equation}
where $z$ is a point in the half-plane ${\mathbb C}_+=\{ z \in {\mathbb C}: \Im z >0 \}$ . Here $f$ is an
analytic function in ${\mathbb C}_+$ represented by the Schwarz formula
\begin{equation} \label{Eq_0.1}
f(z) = \frac{1}{\pi i}\int _{\infty}^{\infty}{\Re f(\zeta) \over \zeta -z}\;d\zeta 
\end{equation}
and such that the boundary values on $\partial{\mathbb C}_+ $ of the real part of $f$
belong to the space $L^p(-\infty, \infty), 1\leq p<\infty$.

Here and in what follows we adopt the notation $||\Re f||_p$ for
$||\Re f |_{\partial{\mathbb C}_+}||_p$, where $|| \cdot ||_p$ stands for the norm in $L^p(-\infty, \infty)$. 
Note that the value $K_{n,\infty}(\alpha)$ 
is obtained by passage to the limit of $K_{n,p}(\alpha)$ 
as $p\rightarrow\infty$. 

Inequality (\ref{Eq_0.2}) with the best possible coefficient in front of $||\Re f||_p$ was obtained by Kresin and Maz'ya
\cite{KM15}. In \cite{KM15} it was shown that
\begin{equation} \label{Eq_0.1M}
K_{n,p}(\alpha)={n! \over \pi}
\left \{ \int_{-\pi/2}^{\pi/2}
\left |\cos \left ( \alpha -(n+1)\varphi+{n\pi \over 2}\right )\right |^q 
\cos^{(n+1)q-2}\varphi d\varphi \right \}^{1/q},
\end{equation}
where $1/p+1/q=1$. So, the sharp constant ${\rm K}_{n,p}$ in the inequality
\begin{equation} \label{Eq_0.2M}
|f^{(n)}(z)| \leq {{\rm K}_{n,p} \over (\Im z)^{n+{1\over p}}}\;||\Re f||_p
\end{equation}
is given by 
\begin{equation} \label{Eq_0.3M}
{\rm K}_{n,p}=\max_\alpha K_{n,p}(\alpha)\;.
\end{equation} 

Note that inequalities (\ref{Eq_0.2}) and (\ref{Eq_0.2M}) for analytic functions belong
to the class of sharp real-part theorems (see Kresin and Maz'ya \cite{KM} and references there)
which go back to Hadamard's real-part theorem \cite{Ha}. 

The present article extends the topic of papers by Kresin and Maz'ya \cite{KM14,KM15}. 
In \cite{KM14} the explicit formulas for ${\rm K}_{0,p}$ for $p\in[1, \infty)$
and for ${\rm K}_{1,p}$ for $p\in[1, \infty]$ were found. In \cite{KM15}
the case of $n\geq 2$ was considered and the explicit formulas for ${\rm K}_{n, p}$ 
were derived for $n=2m+1, 2, 4$ and $p=\infty$ as well as for arbitrary $n$ and $p=1, 2$. Namely, in \cite{KM15} it was shown that  
\begin{equation} \label{Eq_M1}
{\rm K}_{2m+1,\infty}={2\over \pi}{\big ((2m+1)!! \big )^2 \over 2m+1}\;,\;\;\;\;\;\;\;m=0, 1, 2,...,
\end{equation}
\begin{equation} \label{Eq_M2}
{\rm K}_{2, \infty}={3\sqrt{3} \over 2 \pi}\;,\;\;\;\;\;\;\;\;\;\;\;\;\;\;
{\rm K}_{4, \infty}={3 \over 4 \pi}\big ( 16+5\sqrt{5}\big )\;,
\end{equation}
and
\begin{equation} \label{Eq_M3}
\hspace{-7mm}{\rm K}_{n,1}={n! \over \pi}\;,\;\;\;\;\;\;\;\;\;\;\;\;\;\;\;\;\;\;\;{\rm K}_{n,2}=\sqrt{(2n)! \over 2^{2n+1} \pi}\;.
\end{equation}

\smallskip
In this paper the optimization problem (\ref{Eq_0.3M}) is solved in a series of cases described below. In these cases we obtain the explicit 
formulas for the sharp constant ${\rm K}_{n,p}$. In a complicated case $n=2m, p=\infty$ we prove a two-sided estimate for 
${\rm K}_{2m,\infty}$. In conclusion, some applications of obtained results to estimates of high order derivatives of analytic functions in subdomains of ${\mathbb C}$ are described.

\smallskip
Now we describe the results of the present paper in more detail. Introduction is 
followed by four sections. The first of them is auxiliary. It concerns the integral
\begin{equation} \label{Eq_1.1AB}
Q_{\mu,n,\gamma}(\beta)=\int_{-\pi/2}^{\pi/2}
\big |\cos \big ( \beta -(n+1)\varphi \big )\big |^\gamma 
\cos^{\;\mu}\varphi d\varphi ,
\end{equation}
depending on the parameter $\beta $. We consider the problem on maximum of $Q_{\mu,n,\gamma}(\beta)$ in $\beta$.
In what follows, by ${\mathbb N}$ we mean the set of the natural numbers and by $[a]$ we denote the integer part of the number $a$.
Assuming that $m, n \in \{0 \}\cup{\mathbb N}$, $m\geq n+1$ and $\gamma > 2\left [ {m \over n+1} \right ]-2$, we prove the equality
$$
\max _{\beta}Q_{2m,n,\gamma}(\beta)=Q_{2m,n,\gamma}(0)=\int_{-\pi/2}^{\pi/2}\big |\cos (n+1)\varphi \big |^\gamma \cos^{2m}\varphi d\varphi
$$
and find the last integral. If $m\leq n$ and $\gamma > -1$, it is shown that the function $Q_{2m,n,\gamma}(\beta)$ 
is independent of $\beta$ and its value is given. 

Concretizing result of Section 1 for (\ref{Eq_0.1M}) with $q=2(m+1)/(n+1)$ and $n\leq 2m+1$, in Section 2 we obtain the 
explicit formula for ${\rm K}_{n, 2(m+1)/ (2m+1-n) }$. 
In particular, the coefficient $K_{n, 2(m+1)/2m+1-n}(\alpha)$ is independent of $\alpha$ for the case $m\leq n$ and
\begin{equation} \label{Eq_0.5M}
{\rm K}_{n, {2(m+1)\over 2m+1-n}}={n! \over \pi} \left \{ {(2m-1)!!\over (2m)!!}B \left ({m+1 \over n+1}+{1\over 2}\;, {1\over 2} \right )
\right \}^{ \!\! n+1 \over 2(m+1)},
\end{equation}
where by $B$ is denoted the Beta-function. We note, that the above-mentioned formulas for ${\rm K}_{2m+1,\infty}$ 
and ${\rm K}_{n, 2}$ are particular cases of (\ref{Eq_0.5M}) for $n=2m+1$ and $n=m$, correspondingly. 

Section 3 is devoted to derivatives of even order in the case  $p=\infty$. First, we solve the optimization problem 
(\ref{Eq_0.3M}) with $n=6, 8$ and $p=\infty$, and find the values of the sharp constants  
\begin{equation} \label{Eq_E6M}
{\rm K}_{6, \infty}=
{105 \sqrt{2}\over 4\pi}\!\left (9\cos{\pi \over 28}+3\cos{3\pi \over 28}+ \cos{5\pi \over 28} \right ),
\end{equation}
\begin{equation} \label{Eq_E7M}
{\rm K}_{8, \infty}={315 \over 8\pi}\left \{ 175+9\sqrt{2} \left ( 17\cos{\pi \over 36}+
9\cos{5\pi \over 36}+11\cos{7\pi \over 36} \right ) \right \}\;.
\end{equation}

Further, using the result of Section 1, we obtain the two-sided estimate
\begin{equation} \label{Eq_E5M}
{2\over \pi}\big ((2m-1)!! \big )^2 < {\rm K}_{2m, \infty }< {2m \over 2m-1}\;{2\over \pi}\big ((2m-1)!! \big )^2\;,
\end{equation}
which leads to the asymptotic formula 
$$
{\rm K}_{2m, \infty }={2\over \pi}\big ((2m-1)!! \big )^2+ O\left ({ ((2m-1)!!)^2 \over 2m-1 }\right )
$$
as $m \rightarrow \infty $.

Let us denote by
$$
L_{2m}={2\over \pi}\big ((2m-1)!! \big )^2\;,\;\;\;\;\;\;\;\;\;\;\;\;U_{2m}={2m \over 2m-1}\;{2\over \pi}\big ((2m-1)!! \big )^2
$$
the values of the lower and upper bounds in two-sided estimate (\ref{Eq_E5M}), correspondingly.
We can compare these bounds with the sharp constant in inequality (\ref{Eq_0.2M}) for $n=2, 4, 6, 8$ and $p=\infty$.
Using (\ref{Eq_M2}), (\ref{Eq_E6M}) and (\ref{Eq_E7M}), we get
\begin{eqnarray*}
\hspace{-5mm}& &{L_2 \over K_{2,\infty}}\approx 0.7698\;,\;\;\;\;\;\;{L_4 \over K_{4,\infty}}\approx 0.8830\;,\;\;\;\;\;\;
{L_6 \over K_{6,\infty}}\approx 0.9204\;,\;\;\;\;\;\;{L_8 \over K_{8,\infty}}\approx 0.9396,\\
\hspace{-5mm}& &\\
\hspace{-5mm}& &{U_2 \over K_{2,\infty}}\approx 1.5396\;,\;\;\;\;\;\;{U_4 \over K_{4,\infty}}\approx 1.2141
\;,\;\;\;\;\;\;{U_6 \over K_{6,\infty}}\approx 1.1045\;,\;\;\;\;\;\;{U_8 \over K_{8,\infty}}\approx 1.0738.
\end{eqnarray*}

In concluding Section 4 we collect some real-part estimates with explicit constants in the majorizing part of inequality
for the modulus of derivatives of analytic functions in subdomains of ${\mathbb C}$.

\setcounter{equation}{0}
\section{The main lemma} 

First we prove the following auxiliary assertion. 

\setcounter{theorem}{0}
\begin{lemma} \label{LL_1} Let $m, n \in \{0 \}\cup{\mathbb N}$. If $m\geq n+1$ and $\gamma > 2\left [ {m \over n+1} \right ]-2$,
then
\begin{eqnarray} 
& &\hspace{-13mm}\max _{\beta}Q_{2m,n,\gamma}(\beta)=Q_{2m,n,\gamma}(0)=\int_{-\pi/2}^{\pi/2}
\big |\cos (n+1)\varphi \big |^\gamma 
\cos^{2m}\varphi d\varphi\label{Eq_1.2AB}\\
\hspace{-13mm}& &\nonumber\\
& &\hspace{-13mm}={(2m-1)!! \over (2m)!!}\!B\!\left ({\gamma \!+\!1\over 2}, {1\over 2} \right ) +
{\pi \over 2^{2m+\gamma -1}(\gamma\!+\!1)}\sum_{j=1}^{\left [{m \over n+1} \right ]}
{ \left (\begin{array}{ll} \;\;\;\;\;\;\;2m \\ m-j(n+1)\!\!\! \end{array} \right ) \over B\!\left ({\gamma \over 2}\!+\!j\!+\!1, 
{\gamma \over 2}\!-\!j\!+\!1 \right )}\;,\label{Eq_1.3AB}
\end{eqnarray}
where by $B$ is denoted the Beta-function.

If $m\leq n$ and $\gamma > -1$,then the function $Q_{2m,n,\gamma}(\beta)$ is independent of $\beta$, and it is given by
\begin{equation} \label{Eq_1.4AB}
Q_{2m,n,\gamma}(\beta)=
{(2m-1)!! \over (2m)!!}B\left ({\gamma \!+\!1\over 2}, {1\over 2} \right )\;.
\end{equation}
\end{lemma}
\begin{proof} 

Making the change of variable $\psi=\beta-(n+1)\varphi$ in (\ref{Eq_1.1AB}) with $\mu=2m$, we obtain
$$
Q_{2m,n,\gamma}(\beta)={1 \over n+1}\int_{\beta-(n+1){\pi\over2}}^{\beta+(n+1){\pi\over2}}
\big |\cos \psi\big |^\gamma\cos^{2m}{\psi-\beta \over n+1}\; d\psi\;.
$$
Since the integrand is $(n+1)\pi$-periodic, it follows that
\begin{eqnarray*}
Q_{2m,n,\gamma}(\beta)&\!\!\!=\!\!\!&{1 \over n+1}\int_{0}^{(n+1)\pi}
\big |\cos \psi\big |^\gamma\cos^{2m}{\psi-\beta \over n+1}\; d\psi\\
& &\\
&\!\!\!=\!\!\!&{1 \over n+1}\sum_{j=0}^n\int_{j\pi}^{(j+1)\pi}
\big |\cos \psi\big |^\gamma\cos^{2m}{\psi-\beta \over n+1}\; d\psi\;.
\end{eqnarray*}
The change of variable $\psi-j\pi=\vartheta$ implies
\begin{equation} \label{Eq_1.5AB}
Q_{2m,n,\gamma}(\beta)={1 \over n+1}\int_0^\pi
\big |\cos \vartheta\big |^\gamma g_{m,n}(\vartheta-\beta)\; d\vartheta\;,
\end{equation}
where
\begin{equation} \label{Eq_1.6AB}
g_{m,n}(\theta)=\sum_{j=0}^n \cos^{2m}{\theta+j\pi \over n+1}\;.
\end{equation}
Since
$$
\cos^{2m}x={1 \over 2^{2m}}\left \{\left (\begin{array}{ll}\!\!\! 2m\!\!\! \\ \hspace{0.1cm}
\!\!\!m\!\!\! \end{array} \right ) + 2 \sum_{k=0}^{m-1}
\left (\begin{array}{ll}\!\!\! 2m\!\!\! \\ \hspace{0.1cm}\!\!\!k\!\!\! \end{array} \right ) 
\cos 2(m-k)x\right \}\;,
$$
we can write (\ref{Eq_1.6AB}) in the form
$$
g_{m,n}(\theta)={n+1 \over 2^{2m}}\left (\begin{array}{ll}\!\!\! 2m\!\!\! \\ 
\hspace{0.1cm}\!\!\!m \!\!\!\end{array} \right )+{1\over 2^{2m-1}}\sum_{k=0}^{m-1}
\left (\begin{array}{ll}\!\!\! 2m\!\!\! \\ \hspace{0.1cm}\!\!\!k \!\!\!\end{array} 
\right )\sum_{j=0}^{n} \cos{2(m-k)(\theta+j\pi) \over n+1}\;.
$$
Putting here $k=m-l\;(l=1,2,\dots, m)$ and taking into account that
$$
{1 \over 2^{2m}}\left (\begin{array}{ll}\!\!\! 2m\!\!\! \\ 
\hspace{0.1cm}\!\!\!m \!\!\!\end{array} \right )={(2m)! \over 2^{2m}(m!)^2}={(2m)! \over \big ( (2m)!!\big )^2}=
{(2m-1)!! \over (2m)!!}\;,
$$
we obtain
\begin{equation} \label{Eq_1.7AB}
g_{m,n}(\theta)={(2m-1)!!(n+1) \over (2m)!!}+{1\over 2^{2m-1}}\sum_{l=1}^{m}
\left (\begin{array}{ll} 2m \\ \hspace{0.1cm}\!\!\!\!\!m-l \!\!\!\end{array} 
\right )\sum_{j=0}^{n} \cos{2 l(\theta+j\pi) \over n+1}\;.
\end{equation} 
Consider the interior sum in (\ref{Eq_1.7AB}). We have
$$
\sum_{j=0}^{n} \cos{2 l(\theta+j\pi) \over n+1}=
\Re \left \{ e^{2l\theta i\over n+1} \sum_{j=0}^n e^{{2l \pi i\over n+1}j} \right \}\;. 
$$
So, for $l\in \{1,\dots, m \}$,
\begin{equation} \label{Eq_1.8AB}
\sum_{j=0}^{n} \cos{2 l(\theta+j\pi) \over n+1}=\Re \left \{ e^{2l\theta i\over n+1}
{1-e^{2l\pi i} \over 1- e^{2l\pi i\over n+1}} \right \}=0\;,
\end{equation}
if $s={l \over n+1}\notin {\mathbb N}$, and
\begin{equation} \label{Eq_1.9AB}
\sum_{j=0}^{n} \cos{2 l(\theta+j\pi) \over n+1}=(n+1)\cos 2s \theta\;,
\end{equation}
if  $s={l \over n+1}\in {\mathbb N}$.

Taking into account (\ref{Eq_1.8AB}) and (\ref{Eq_1.9AB}), we can rewrite (\ref{Eq_1.7AB}) in the form
$$
g_{m,n}(\theta)={(2m-1)!!(n+1) \over (2m)!!}+{n+1 \over 2^{2m-1}}\sum_{s=1}^{\left [{m \over n+1} \right ]}
\left (\begin{array}{ll} \;\;\;\;\;2m \\ \hspace{0.1cm}\!\!\!\!\!m-s(n+1) \!\!\!\end{array} 
\right )\cos 2s \theta\;.
$$ 
Combining this with (\ref{Eq_1.5AB}), we obtain
\begin{eqnarray*}  
& &Q_{2m,n,\gamma}(\beta)={(2m-1)!! \over (2m)!!}\int_0^\pi
\big |\cos \vartheta\big |^\gamma\; d\vartheta\\
& &\\
& &+{1 \over 2^{2m-1}}\sum_{s=1}^{\left [{m \over n+1} \right ]}
\left (\begin{array}{ll} \;\;\;\;\;2m \\ \hspace{0.1cm}\!\!\!\!\!m-s(n+1) \!\!\!\end{array} 
\right )\int_0^\pi
\big |\cos \vartheta\big |^\gamma\cos 2s (\vartheta -\beta)\; d\vartheta\;.
\end{eqnarray*}
Since the integrands in the last equality are $\pi$-periodic, it follows that
\begin{eqnarray*}
\hspace{-12mm}& &Q_{2m,n,\gamma}(\beta)={(2m-1)!! \over (2m)!!}\int_{-\pi/2}^{\pi/2}
\cos ^\gamma \vartheta \; d\vartheta\\
\hspace{-12mm}& &\\
\hspace{-12mm}& &+{1 \over 2^{2m-1}}\sum_{s=1}^{\left [{m \over n+1} \right ]}
\left (\begin{array}{ll} \;\;\;\;\;2m \\ \hspace{0.1cm}\!\!\!\!\!m-s(n+1) \!\!\!\end{array} 
\right )\int_{-\pi/2}^{\pi/2}
\cos ^\gamma\vartheta \cos 2s (\vartheta -\beta)\; d\vartheta\;,
\end{eqnarray*}
that is
\begin{eqnarray} \label{Eq_1.10AB}
\hspace{-12mm}& &Q_{2m,n,\gamma}(\beta)={2(2m-1)!! \over (2m)!!}\int_{0}^{\pi/2}
\cos ^\gamma \vartheta
\; d\vartheta\nonumber\\
\hspace{-12mm}& &\nonumber\\
\hspace{-12mm}& &+{1 \over 2^{2(m-1)}}\sum_{s=1}^{\left [{m \over n+1} \right ]}
\left (\begin{array}{ll} \;\;\;\;\;2m \\ \hspace{0.1cm}\!\!\!\!\!m-s(n+1) \!\!\!\end{array} 
\right )\cos 2s \beta \int_{0}^{\pi/2}
\cos ^\gamma \vartheta
\cos 2s\vartheta \; d\vartheta\;.
\end{eqnarray}
Let $m\geq n+1$. Taking into account the formula (see, e.g., Gradshtein and Ryzhik \cite{GRJ}, {\bf 3.631(9)})
\begin{equation} \label{Eq_1.11AB}
\int_0^{\pi/2}\cos^{\nu -1}x \cos ax dx={\pi \over 2^\nu \nu B\left ({\nu+a+1 \over 2}, {\nu-a+1 \over 2}\right )}\;,
\end{equation}
where $\Re \;\nu >0$, and the condition $\gamma > 2\left [ {m \over n+1} \right ]-2$ of the present lemma, we conclude that
$$
\int_{0}^{\pi/2}\cos ^\gamma\vartheta \cos 2s\vartheta \; d\vartheta >0
$$
for any $s \in \left \{1, 2, \dots, \left [{m \over n+1} \right ] \right \}$. This and (\ref{Eq_1.10AB}) imply that
the maximum of $Q_{2m,n,\gamma}(\beta)$ in $\beta$ is attained at $\beta=0$. Hence, by (\ref{Eq_1.1AB}) we obtain (\ref{Eq_1.2AB}).
Calculating by (\ref{Eq_1.11AB}) the integrals in (\ref{Eq_1.10AB}), we arrive at (\ref{Eq_1.3AB}).

The sum in (\ref{Eq_1.10AB}) vanishes in the case $m\leq n$. Hence, the function $Q_{2m,n,\gamma}(\beta)$ 
is independent of $\beta$ under condition $m\leq n$, which proves (\ref{Eq_1.4AB}).
\end{proof}

\setcounter{equation}{0}
\section{Sharp estimates for derivatives of analytic 
functions with $\Re f \in h^p({\mathbb R}^2_+)$, $p=2(m+1)/ (2m+1-n)$}
 
In what follows, by $h^p({\mathbb R}^2_+), 1\leq p\leq\infty$, we mean
the Hardy space of harmonic functions in the upper half-plane ${\mathbb R}^2_+$ 
which are represented by the Poisson integral  with a density in $L^p(-\infty, \infty)$.
It is well known (see, e.g. Levin \cite{Le}, Sect. 19.3) that $f$ belongs to the Hardy
space $H^p({\mathbb C}_+)$ of analytic functions in ${\mathbb C}_+$ if $\Re f \in 
h^p({\mathbb R}^2_+), 1 < p<\infty$. 
Besides, any function $f\in H^p({\mathbb C}_+), 1<p<\infty$, admits the representation
(\ref{Eq_0.1}) since $\Re f\in h^p({\mathbb R}^2_+)$.
     
Now we consider the case $p=2(m+1)/ (2m+1-n)$ in inequality (\ref{Eq_0.2M}), that is $q=2(m+1)/ (n+1)$. We suppose that
$n\leq 2m+1$, $n\geq 1$. In the case $n= 2m+1$ we put $p=\infty$.

\setcounter{theorem}{0}
\begin{theorem} \label{PP_1}
Let $\Re f \in h^p({\mathbb R}^2_{+})$ with $p=2(m+1)/ (2m+1-n)$, and let  $z $ be  an arbitrary point in 
${\mathbb C}_+$. The sharp constant ${\rm K}_{n, 2(m+1)/ (2m+1-n)}$ in the inequality
\begin{equation} \label{Eq_3.1}
|f^{(n)}(z)| \leq {{\rm K}_{n,\; 2(m+1)/ (2m+1-n)} \over (\Im z)^{n+{1\over p}}}\;||\Re f||_p
\end{equation}
is given by
\begin{equation} \label{Eq_3.2}
{\rm K}_{n, {2(m+1)\over 2m+1-n} }\!=\!{n! \over \pi}\left \{ \int_{-\pi/2}^{\pi/2}
\big |\cos (n+1)\varphi \big |^{2(m+1) \over n+1} 
\cos^{2m}\varphi d\varphi \right \} ^{n+1 \over 2(m+1)}
\end{equation}
\begin{eqnarray}
& &\hspace{-8mm}\!=\!{n!\over \pi}  \left \{ {(2m-1)!! \over (2m)!!}\;B\!\left ({m+1 \over n+1}+{1\over 2}, {1\over 2} \right )\right .\nonumber\\
& &\hspace{-8mm}\label{Eq_3.3}\\
& &\hspace{-8mm}\left . +{\pi (n+1) \over 2^{2m-1+{2(m+1) \over n+1}}\big ( 2m+n+3\big )}\sum_{j=1}^{\left [{m \over n+1} \right ]}
{ \left (\begin{array}{ll} \;\;\;\;\;\;\;2m \\ m-j(n+1)\!\!\! \end{array} \right ) \over B\!\left ({m+1 \over n+1}+j+1, {m+1 \over n+1}-j+1 \right )}\right \}^{ \!\! n+1 \over 2(m+1)}\!\!\!\!\!\!.\nonumber
\end{eqnarray}

In particular, 
\begin{equation} \label{Eq_3.4}
{\rm K}_{n, {2(m+1)\over 2m+1-n}}={n! \over \pi} \left \{ {(2m-1)!!\over (2m)!!}B \left ({m+1 \over n+1}+{1\over 2}\;, {1\over 2} \right )
\right \}^{ \!\! n+1 \over 2(m+1)}
\end{equation}
for $m\leq n$.
\end{theorem}
\begin{proof} Putting $q=2(m+1)/ (n+1)$ in (\ref{Eq_0.1M})-(\ref{Eq_0.3M}) and $\gamma=2(m+1)/ (n+1)$, $\mu=2m$ in (\ref{Eq_1.1AB}), we can write
the sharp constant ${\rm K}_{n,p}$ in inequality (\ref{Eq_3.1}) as follows 
\begin{eqnarray} \label{Eq_3.5}
\!{\rm K}_{n,{2(m+1)\over 2m+1-n}}\!&=&\!{n! \over \pi}\max_{\alpha}\left \{ Q_{2m,n, {2(m+1)\over n+1}}
\left (\alpha +{n\pi \over 2}\right ) \right \}^{n+1 \over 2(m+1)}\nonumber\\
& &\nonumber\\
\!& =&\!{n! \over \pi}\max _{\beta}\big \{ Q_{2m,n,{2(m+1)\over n+1}}\left (\beta \right ) \big \}^{n+1 \over 2(m+1)}.
\end{eqnarray}
For $\gamma=2(m+1)/ (n+1)$ we have
$$
{\gamma \over 2}+1={m+1 \over n+1}+1> \left [ {m \over n+1}\right ]+1 
$$
that is, the condition 
$$
\gamma > 2\left [ {m \over n+1}\right ]-2
$$
of Lemma \ref{LL_1} is satisfied. Applying Lemma \ref{LL_1} to (\ref{Eq_3.5}), we complete the proof.
\end{proof}

Two following consequences of Theorem \ref{PP_1} contain explicit formulas for ${\rm K}_{n,p}$ in particular cases.

\setcounter{theorem}{0}
\begin{corollary} \label{CR_1}  If $n=2m$, $m \in {\mathbb N}$, then
\begin{equation} \label{Eq_3.6}
{\rm K}_{2m, 2m+2}={(2m)!\over \pi} \left \{ {\sqrt{\pi}(2m-1)!!\;\Gamma \left ({m+1 \over 2m+1}+{1\over 2}\right ) 
\over (2m)!!\Gamma \left ({m+1 \over 2m+1}+1 \right )}\right \}^{2m+1 \over 2(m+1)}\;.
\end{equation}
\end{corollary}

\begin{corollary} \label{CR_3}  If $m=k(n+1)-1$, $k \in {\mathbb N}$, then
\begin{eqnarray*} 
{\rm K}_{n,{2k \over 2k-1}}&\!\!\!=\!\!\!&
{n!\over \pi} \left \{ {\sqrt{\pi}\big (2k(n+1)-3 \big )!!\;\Gamma\!\left (k+{1\over 2}\right ) \over \big (2k(n+1)-2 \big )!!k!}\right . \\
& &\\
&\!\!\!+\!\!\!&\left .{\pi \over 2^{2k(n+2)-3}}\sum_{j=1}^{k-1}
\left (\begin{array}{ll} \!\!2k(n+1)-2 \\ (k-j)(n+1)-1 \!\!\end{array} \right )
\left (\begin{array}{ll} \;\;2k \\k-j\!\! \end{array} \right ) \right \}^{ 1 \over 2k}\!\!\!\!\!\!.
\end{eqnarray*}
\end{corollary}

\setcounter{equation}{0}
\section{Estimates for even order derivatives of analytic 
functions with $\Re f \in h^\infty({\mathbb R}^2_+)$}

By (\ref{Eq_0.1M}),
\begin{equation} \label{Eq_1.3HO}
K_{2m,\infty}(\alpha)={(2m)! \over \pi}
 \int_{-\pi/2}^{\pi/2}
\left |\cos \left ( \alpha -(2m+1)\varphi\right )\right | 
\cos^{2m-1}\varphi d\varphi .
\end{equation}

The starting point of this section is the following assertion from the paper by Kresin and Maz'ya \cite{KM15}.
\begin{lemma} \label{LL_4} The equality
\begin{equation} \label{Eq_4.3HO}
{d K_{2m, \infty} \over d\alpha}\!=\!{(2m)! \over \pi(2m+1)^2 2^{2(m-1)}}\!\int_0^{\pi/2}\!\!\!
\big ( |\cos(\alpha-\varphi)|-|\cos(\alpha+\varphi)|\big )\Lambda_m(\varphi)d\varphi
\end{equation}
holds with
\begin{equation} \label{Eq_4.4HO}
\Lambda_m(\varphi)\!=\!\sum_{\ell=1}^m(-1)^{\ell}(2\ell-1)
\left (\begin{array}{ll}\!\!\! 2m-1\!\!\! \\ 
\hspace{0.1cm}\!\!\!m-\ell \!\!\!\end{array} \right )
{\sin {(2\ell-1)\varphi \over 2m+1}\over 
\sin {(2\ell-1)\pi \over 2(2m+1)}}\;.
\end{equation}
\end{lemma}

{\bf Remark 1.} Before passing to applications of Lemma \ref{LL_4} we make two remarks. 
The first one concerns the range of $\beta$ in the evaluation of the maximum
$$
\max_\beta Q_{\mu,n,\gamma}(\beta),
$$
where the function $Q_{\mu,n,\gamma}(\beta)$ is defined by (\ref{Eq_1.1AB}).
It is clear, that $Q_{\mu,n,\gamma}(\beta)$ is $\pi$-periodic and even function in $\beta$. 
Therefore, we can limit our consideration of $Q_{\mu,n,\gamma}(\beta)$ to the interval $[0, \pi/2]$.

The second remark relates the sign of the function $|\cos(\beta -\varphi)|\!-\!|\cos(\beta +\varphi)|$, 
which appear inside of integral (\ref{Eq_4.3HO}). We show that
\begin{equation} \label{Eq_4.12AB}
|\cos(\beta -\varphi)|\geq|\cos(\beta +\varphi)|
\end{equation}
for $\beta, \varphi \in [0, \pi/2]$. In fact, since
\begin{eqnarray*}
|\cos(\beta -\varphi)|\! -\!|\cos(\beta +\varphi )|\! =\!\left \{
\begin{array}{lll}\displaystyle{
   \cos(\beta -\varphi)\! -\!\cos(\beta +\varphi)}   
   & \!\!\!\quad\hbox {for}\!\quad \varphi \in \left [\;0, {\pi\over 2}-\beta \right ],\\
     & \\ \displaystyle{
      \cos(\beta -\varphi )\! +\!\cos(\beta +\varphi )}& \!\!\!\quad\hbox {for}\!\quad 
      \varphi \in \left ({\pi\over 2}-\beta , {\pi \over 2} \right ],
\end{array}\right .
\end{eqnarray*}
it follows that
\begin{eqnarray*} 
|\cos(\beta -\varphi )|\! -\!|\cos(\beta +\varphi )|\! =\!\left \{
\begin{array}{lll}\displaystyle{
   2\sin\varphi \sin \beta}   
   & \quad\hbox {for}\quad \varphi \in \left [\;0, {\pi\over 2}-\beta \right ]\;,\\
     & \\ \displaystyle{
      2\cos\varphi \cos\beta }& \quad\hbox {for}\quad 
      \varphi \in \left ({\pi\over 2}-\beta , {\pi \over 2} \right ] \;,
\end{array}\right .
\end{eqnarray*}
and hence (\ref{Eq_4.12AB}) holds for $\beta, \varphi \in [0, \pi/2]$. 
Besides, the equality sign in (\ref{Eq_4.12AB}) holds 
only for $\beta=0$ or for $\beta=\pi/2$ provided that $\varphi\in(0, \pi/2)$.

\smallskip
In the next two assertions we deal with the values of constants ${\rm K}_{6,\infty}$ and ${\rm K}_{8,\infty}$.

\setcounter{theorem}{5}
\begin{corollary} \label{CO_6}
Let $\Re f \in h^\infty({\mathbb R}^2_{+})$, and let  $z $ be  an arbitrary point in 
${\mathbb C}_+$. The sharp constant ${\rm K}_{6, \infty} $ in the inequality
\begin{equation} \label{Eq_4.16HO}
|f^{(6)}(z)| \leq {{\rm K}_{6,\infty} \over (\Im z)^6}\;||\Re f||_\infty
\end{equation}
is given by
\begin{equation} \label{Eq_6.7HO}
{\rm K}_{6, \infty}=
{105 \sqrt{2}\over 4\pi}\!\left (9\cos{\pi \over 28}+3\cos{3\pi \over 28}+ \cos{5\pi \over 28} \right ).
\end{equation}
\end{corollary}
\begin{proof} By Lemma \ref{LL_4},
\begin{equation} \label{Eq_4.18HO}
{d K_{6, \infty} \over d\alpha}\!=\!{45 \over 49\pi}\!\int_0^{\pi/2}\!\!\!
\big ( |\cos(\alpha-\varphi)|-|\cos(\alpha+\varphi)|\big )\Lambda_3(\varphi)\;d\varphi\;,
\end{equation}
where
{\large 
\begin{equation} \label{Eq_4.15HO}
\Lambda_3(\varphi)=5\left (-2{\sin{\varphi \over 7}\over \sin{\pi \over 14} }+
3{\sin{3\varphi \over 7} \over \sin{3\pi \over 14}} - {\sin{5\varphi \over 7} \over \sin{5\pi \over 14}}\right ).
\end{equation}
}
Using the identities $\sin 3x=3\sin x-4\sin^3 x, \sin 5x=5\sin x-20\sin^3 x+16\sin^5 x$ in (\ref{Eq_4.15HO}), we find
\begin{equation} \label{Eq_4.15HO1}
\Lambda_3(\varphi)=80{\sin{\pi \over 14}\sin{\varphi \over 7}  \over \sin{3\pi \over 14}\sin{5\pi \over 14}}
\left ( \sin^2{\pi \over 14} - \sin^2{\varphi \over 7} \right )F_3(\varphi)\;,
\end{equation}
where
$$
F_3(\varphi)=\left (8\sin^2{\pi \over 14} -7 \right )\sin^2{\pi \over 14}+\left (3-4\sin^2{\pi \over 14} \right )\sin^2{\varphi \over 7}\;.
$$
Since $3-4\sin^2(\pi / 14)>3-4\sin^2(\pi / 6)>0$, we have
$$
F_3(\varphi)<\left (8\sin^2{\pi \over 14} -7 \right )\sin^2{\pi \over 14}+\left (3-4\sin^2{\pi \over 14} \right )\sin^2{\pi \over 14}=
-4\cos^2{\pi \over 14}\sin^2{\pi \over 14},
$$
which together with (\ref{Eq_4.15HO1}) proves the inequality $\Lambda_3(\varphi)<0$ for $\varphi \in (0, \pi /2)$.
Now, by (\ref{Eq_4.18HO}) and (\ref{Eq_4.12AB}) we conclude that
$$
{d K_{6, \infty} \over d\alpha}<0
$$
for $\alpha\in (0, \pi/2)$. Thus, by (\ref{Eq_1.3HO}),
\begin{eqnarray*}
\hspace{-7mm}& &{\rm K}_{6, \infty}=K_{6, \infty}(0)={6! \over \pi}
\int_{-\pi/2}^{\pi/2}\big |\cos 7\varphi\big |\cos^5 \varphi \;d\varphi=
2{6!\over \pi}\left \{\!\int_{0}^{\pi/14}\!\!\!\!\cos 7\varphi\cos^5 \varphi \right . d\varphi\\
\hspace{-7mm}& &\\
\hspace{-7mm}& &\left .
\!-\!\int_{\pi/14}^{3\pi/14}\!\!\!\!\cos 7\varphi\cos^5 \varphi \;d\varphi \!+\!
\int_{3\pi/14}^{5\pi/14}\!\!\!\!\cos 7\varphi\cos^5 \varphi \;d\varphi
-\int_{5\pi/14}^{\pi/2}\!\!\!\!\cos 7\varphi\cos^5 \varphi \;d\varphi\right \}.
\end{eqnarray*}
Evaluating the integrals on the right-hand side of the last equality, 
we arrive at (\ref{Eq_6.7HO}).
\end{proof}

\setcounter{theorem}{6}
\begin{corollary} \label{CO_7}
Let $\Re f \in h^\infty({\mathbb R}^2_{+})$, and let  $z $ be  an arbitrary point in 
${\mathbb C}_+$. The sharp constant ${\rm K}_{8, \infty} $ in the inequality
\begin{equation} \label{Eq_4.16HO8}
|f^{(8)}(z)| \leq {{\rm K}_{8,\infty} \over (\Im z)^8}\;||\Re f||_\infty
\end{equation}
is given by
\begin{equation} \label{Eq_6.7HO8}
{\rm K}_{8, \infty}={315 \over 8\pi}\left \{ 175+9\sqrt{2} \left ( 17\cos{\pi \over 36}+
9\cos{5\pi \over 36}+11\cos{7\pi \over 36} \right ) \right \}\;.
\end{equation}
\end{corollary}
\begin{proof} By Lemma \ref{LL_4},
\begin{equation} \label{Eq_4.18HO8}
{d K_{8, \infty} \over d\alpha}\!=\!{70 \over 9\pi}\!\int_0^{\pi/2}\!\!\!
\big ( |\cos(\alpha-\varphi)|-|\cos(\alpha+\varphi)|\big )\Lambda_4(\varphi)\;d\varphi\;,
\end{equation}
where
{\large 
\begin{equation} \label{Eq_4.15HO8}
\Lambda_4(\varphi)=7\left (-5{\sin{\varphi \over 9}\over \sin{\pi \over 18} }+
18\sin{3\varphi \over 9} -5 {\sin{5\varphi \over 9} \over \sin{5\pi \over 18}}+
{\sin{7\varphi \over 9} \over \sin{7\pi \over 18}}\right ).
\end{equation}
}
Using the identities $\sin 3x=3\sin x-4\sin^3 x, \sin 5x=5\sin x-20\sin^3 x+16\sin^5 x$ and 
$\sin 7x=7\sin x-56\sin^3 x+112\sin^5 x- 64\sin^7 x$ in (\ref{Eq_4.15HO}), we find
\begin{equation} \label{Eq_4.15HO18}
\Lambda_4(\varphi)=896{\sin^2{\pi \over 18}\sin{\varphi \over 9}  \over \sin{5\pi \over 18}\sin{7\pi \over 18}}
\left ( \sin^2{\pi \over 18} - \sin^2{\varphi \over 9} \right )F_4(\varphi)\;,
\end{equation}
where
\begin{eqnarray*}
\hspace{-7mm}& &F_4(\varphi)=\left (155-604\sin^2{\pi \over 18} +768\sin^4{\pi \over 18} -320\sin^6{\pi \over 18}\right )\sin^4{\pi \over 18}\\
\hspace{-7mm}& &\\
\hspace{-7mm}& &+\left (-90+396\sin^2{\pi \over 18}-560\sin^4{\pi \over 18}+256\sin^6{\pi \over 18} \right )
\sin^2{\pi \over 18}\sin^2{\varphi \over 9}\\
\hspace{-7mm}& &\\
\hspace{-7mm}& &+\left (15-80\sin^2{\pi \over 18}+128\sin^4{\pi \over 18}-64\sin^6{\pi \over 18} \right )\sin^4{\varphi \over 9}\;.
\end{eqnarray*}
It follows
\begin{eqnarray*}
\hspace{-7mm}& &F_4(\varphi)>\left (155-604\sin^2{\pi \over 18} +768\sin^4{\pi \over 18} -320\sin^6{\pi \over 18}\right )\sin^4{\pi \over 18}\\
\hspace{-7mm}& &\\
\hspace{-7mm}& &-10\left (9+56\sin^4{\pi \over 18} \right )\sin^4{\pi \over 18}
-16\left (5+4\sin^4{\pi \over 18} \right )\sin^6{\pi \over 18}\\
\hspace{-7mm}& &\\
\hspace{-7mm}& &=\left (65-684\sin^2{\pi \over 18}+208 \sin^4{\pi \over 18}-384 \sin^6{\pi \over 18}\right )\sin^4{\pi \over 18}\;.
\end{eqnarray*}
Using this inequality and $\sin(\pi/12)=\big (\sqrt{6}-\sqrt{2}\; \big )/4$, we obtain
\begin{eqnarray*}
\hspace{-7mm}& &F_4(\varphi)>\left (65-684\sin^2{\pi \over 18}-384 \sin^6{\pi \over 18}\right )\sin^4{\pi \over 18}\\
\hspace{-7mm}& &\\
\hspace{-7mm}& &>\left (65-684\sin^2{\pi \over 12}-384 \sin^6{\pi \over 12}\right )\sin^4{\pi \over 18}=
\left (261\sqrt{3}-433\right )\sin^4{\pi \over 18}>0\;.
\end{eqnarray*}
The last estimate together with (\ref{Eq_4.15HO18}) proves the inequality $\Lambda_4(\varphi)>0$ for $\varphi \in (0, \pi /2)$.
Now, by (\ref{Eq_4.18HO8}) and (\ref{Eq_4.12AB}) we conclude that
$$
{d K_{8, \infty} \over d\alpha}>0
$$
for $\alpha\in (0, \pi/2)$. Thus, by (\ref{Eq_1.3HO}),
\begin{eqnarray*}
{\rm K}_{8, \infty}=K_{8, \infty}(\pi/2)&\!\!=\!\!&{8! \over \pi}
\int_{-\pi/2}^{\pi/2}\big |\sin 9\varphi\big |\cos^7 \varphi \;d\varphi=
2{8!\over \pi}\left \{\!\int_{0}^{\pi/9}\!\!\!\!\sin 9\varphi\cos^7 \varphi \right . d\varphi\\
& &\\
& &\!-\!\int_{\pi/9}^{2\pi/9}\!\!\!\!\sin 9\varphi\cos^7 \varphi\;d\varphi +
\int_{2\pi/9}^{\pi/3}\!\!\!\!\sin 9\varphi\cos^7 \varphi\;d\varphi\\
& &\\
& &\left . -\int_{\pi/3}^{4\pi/9}\!\!\!\!\sin 9\varphi\cos^7 \varphi \;d\varphi +
\int_{4\pi/9}^{\pi/2}\!\!\!\!\sin 9\varphi\cos^7 \varphi \;d\varphi\right \}\;.
\end{eqnarray*}
After evaluating the integrals on the right-hand side of the last equality, 
we arrive at (\ref{Eq_6.7HO8}).
\end{proof}

Now we apply Lemma \ref{LL_1} in the proof of the following assertion.

\setcounter{theorem}{1} 
\begin{theorem} \label{TR_1} The following two-side inequality
\begin{equation} \label{Eq_4.1AB}
{2\over \pi}\big ((2m-1)!! \big )^2 
< {\rm K}_{2m, \infty }< {2m \over 2m-1}\;{2\over \pi}\big ((2m-1)!! \big )^2
\end{equation}
holds.
\end{theorem}
\begin{proof} By (\ref{Eq_1.3HO}),
$$
K_{2m,\infty}(\alpha)<{(2m)! \over \pi} \int_{-\pi/2}^{\pi/2}
\big |\cos \big ( \alpha -(2m+1)\varphi\big )\big | \cos^{2(m-1)}\varphi\; d\varphi\;.
$$
From this and equality (\ref{Eq_1.4AB}) with $m-1$ instead of $m$ and $n=2m, \gamma=1$, we obtain
\begin{equation} \label{Eq_4.6AB}
{\rm K}_{2m,\infty}(\alpha) < {2(2m)! \over \pi}{(2m-3)!! \over (2m-2)!!}={2\over \pi}{2m \over 2m-1}{(2m-1)! (2m-1)!! \over (2m-2)!!}\;,
\end{equation}
which together with (\ref{Eq_0.3M}) proves the upper estimate in (\ref{Eq_4.1AB}).

Now we turn to the inverse estimate of ${\rm K}_{2m,\infty}$ in (\ref{Eq_4.1AB}). It follows from (\ref{Eq_0.3M}) and (\ref{Eq_1.3HO}),
$$
{\rm K}_{2m,\infty}\geq K_{2m,\infty}(\alpha)>{(2m)! \over \pi} \int_{-\pi/2}^{\pi/2}
\big |\cos \big ( \alpha -(2m+1)\varphi\big )\big |\cos^{2m}\varphi\; d\varphi \;.
$$
Using the last estimate and (\ref{Eq_1.4AB}) with $n=2m$ and $\gamma=1$, we find that
$$
{\rm K}_{2m,\infty}> 2{(2m)! \over \pi}\;{(2m-1)!! \over (2m)!!}\;,
$$
which is equivalent to the lower estimate in (\ref{Eq_4.1AB}). 
\end{proof}

In final section we describe some real-part estimates which take the explicit form by combination with
formulas for ${\rm K}_{n, p}$ from (\ref{Eq_M1})-(\ref{Eq_M3}), Theorem \ref{PP_1}, 
Corollaries \ref{CR_1}-\ref{CO_7} and the estimate of Theorem \ref{TR_1}. 

\setcounter{equation}{0}
\section{Explicit estimates for derivatives of analytic functions in domains}

The next two assertions were proved in paper by Kresin and Maz'ya \cite{KM16} .
\setcounter{theorem}{0}

\setcounter{theorem}{0}
\begin{proposition} \label{T_1C} 
Let $\Omega={\mathbb C}\backslash {\overline G}$, where $G$ is a convex domain in ${\mathbb C}$,
and let $f$ be a holomorphic function in $\Omega$ with bounded real part.
Then for any point $z\in \Omega$ the inequality
$$
\big |f^{(n)}(z)\big |\leq {{\rm K}_{n, \infty} \over d_z^n}\sup_\Omega |\Re f|\;,\;\;\;\;\;\;\;\;n=1,2,\dots\;, 
$$
holds with $d_z={\rm dist}\;(z, \partial \Omega)$, where
$$
{\rm K}_{n, \infty}={n! \over \pi}\max_{\beta}\int_{-\pi/2}^{\pi/2} \big |\cos \big (\beta - (n+1)
\varphi \big ) \big |\cos ^{n-1}\varphi\;d\varphi
$$
is the best constant in the inequality
$$
|f^{(n)}(z)| \leq {{\rm K}_{n, \infty} \over (\Im z)^n }\;||\Re f||_\infty
$$
for holomorphic functions $f$ in the half-plane ${\mathbb C}_+$ with the bounded real part.
\end{proposition}

\begin{proposition} \label{T_2C} Let $\Omega$ be a domain in ${\mathbb C}$,
and let ${\mathfrak R}(\Omega)$ be the set of holomorphic functions $f$ in $\Omega$
with $\sup_\Omega|\Re f|\leq 1$.
Assume that a point $\zeta \in \partial\Omega$ can be touched by an interior disk $D$. Then
$$
\limsup _{z\rightarrow \zeta} \sup_{f\in {\mathfrak R}(\Omega)}|z-\zeta|^n\big |f^{(n)}(z)\big | \leq 
{\rm K}_{n, \infty}\;,\;\;\;\;\;\;\;\;n=1,2,\dots\;, 
$$
where $z$ is a point of the radius of $D$ directed from the center to $\zeta$.
Here the constant ${\rm K}_{n, \infty}$ is the same as in Proposition $\ref{T_1C}$ and cannot be diminished.
\end{proposition}

By (\ref{Eq_M1}) and (\ref{Eq_4.1AB}), the constant ${\rm K}_{n, \infty}$ in Propositions \ref{T_1C} and \ref{T_2C} obeys
the relations
$$
{\rm K}_{2m-1, \infty}={2\over \pi}{\big ((2m-1)!! \big )^2 \over 2m-1}\;,\;\;\;\;\;\;\;\;\;\;\;
{\rm K}_{2m, \infty}<{2m \over 2m-1}\;{2\over \pi}\big ((2m-1)!! \big )^2
$$
for any $m\in {\mathbb N}$. The values of the constants ${\rm K}_{2, \infty}\;$, ${\rm K}_{4, \infty}\;$, 
${\rm K}_{6, \infty}$ and ${\rm K}_{8, \infty}$ in these statements are given by (\ref{Eq_M2}), (\ref{Eq_6.7HO}) 
and (\ref{Eq_6.7HO8}), correspondingly. 

\smallskip
Now we turn to a real-part estimate for derivatives of analytic functions in the disk ${\mathbb D}=\{ z\in {\mathbb C}: |z|<1 \}$.
By $h^p({\mathbb D})$, $1\leq p\leq\infty$, we mean the Hardy space of harmonic functions in the real unit disk
${\mathbb D}$ which are represented by the Poisson integral with a density in $L^p(\partial {\mathbb D})$. 
Below by $||\cdot||_p$ we denote the norm in the space $L^p(\partial {\mathbb D})$.

The inequality, obtained by Khavinson \cite{KHAV} 
$$
|f'(z)|\leq {4 \over \pi(1-r^2)} ||\Re f||_\infty\;,
$$
contains the best possible coefficient in front of $||\Re f||_\infty $, where $r=|z|<1$. 

The next estimate for derivatives of analytic functions with $\Re f \in h^p({\mathbb D})$
\begin{equation} \label{Eq_6.3M}
|f^{(n)}(z)|\leq {C_{n, p} \over (1-r^2)^{n+{1\over p}}} ||\Re f||_p
\end{equation}
was proved in the paper by Kalaj and Elkies \cite{KAEL}.
The representation of the constant $2^{- \left ( n+{1\over p} \right )}C_{n, p}$ in \cite{KAEL} is equivalent 
to the representation (\ref{Eq_0.3M}) with (\ref{Eq_0.1M}) for the sharp constant ${\rm K}_{n, p}$ in inequality (\ref{Eq_0.2M}).
The case $n=1$ in (\ref{Eq_6.3M}) was considered by Kalaj and Markovi\'c \cite{KAMA3}.

\smallskip
The assertion below was established in paper by Kresin \cite{KM17}. 
\begin{proposition} \label{COR_1A} Let $f$ be an analytic function in ${\mathbb D}$ with $\Re f \in h^p({\mathbb D})$. 
The inequality
$$
\sup_{|z|<1}\;\sup _{||\Re f||_p\leq 1}\;(1-|z|^2)^{n+{1\over p}}\;|f^{(n)}(z)|\geq 
2^{n+{1\over p}}{\rm K}_{n, p}
$$
holds, where ${\rm K}_{n, p}$ is the sharp constant in inequality $(\ref{Eq_0.2M})$.
\end{proposition}
 
Proposition \ref{COR_1A}  together with (\ref{Eq_6.3M}) leads to relation
$$
\sup_{|z|<1}\;\sup _{||\Re f||_p\leq 1}\;(1-|z|^2)^{n+{1\over p}}\;|f^{(n)}(z)|=2^{n+{1\over p}}{\rm K}_{n, p}\;,
$$
which shows that the constant $C_{n, p}=2^{n+{1\over p}}{\rm K}_{n, p}$ 
in estimate (\ref{Eq_6.3M}) cannot be diminished.

The explicit expression for $C_{2m-1, \infty}$ was established by Kalaj and Elkies \cite{KAEL}. 
The formulas for $C_{2m, \infty}=2^{2m}{\rm K}_{2m, \infty}$ with $m=1, 2, 3, 4$ can be obtained by 
(\ref{Eq_M2}), (\ref{Eq_6.7HO}) and (\ref{Eq_6.7HO8}). Other examples of the explicit formulas for the constant $C_{n, p}$ 
in (\ref{Eq_6.3M}) can be derived by relation $C_{n, p}=2^{n+{1\over p}}{\rm K}_{n, p}$
and Theorem \ref{PP_1} as well as Corollaries \ref{CR_1}, \ref{CR_3}.

The next two-sided inequality 
$$
{2^{2m+1}\over \pi}\big ((2m-1)!! \big )^2 
< C_{2m, \infty }< {2m \over 2m-1}\;{2^{2m+1}\over \pi}\big ((2m-1)!! \big )^2
$$
follows from equality $C_{2m, \infty}=2^{2m}{\rm K}_{2m, \infty}$ and estimate (\ref{Eq_4.1AB}). This implies
$$
C_{2m, \infty }\sim {2^{2m+1}\over \pi}\big ((2m-1)!! \big )^2
$$
as $m\rightarrow \infty $.


\begin{thebibliography}{99}

\bibitem{GRJ} I.S. Gradshtein and I.M. Ryzhik; A. Jeffrey, editor,
{\it Table of Integrals, Series and Products}, Fifth edition, Academic Press, New York, 1994.

\bibitem{Ha} J. Hadamard, \textit{Sur les fonctions enti\`eres de la forme $e^{G(X)}$},
C.R. Acad. Sci., \textbf{114} (1892), 1053--1055.

\bibitem{KAEL} D. Kalaj and N. D. Elkies, \textit{On real part theorem for the higher derivatives of analytic functions in the unit disk}, 
Comp. Meth. and Function Theory, {\bf 13}:2 (2013), 189--203.

\bibitem{KAMA3} D. Kalaj and M. Markovi\'c, \textit{Optimal estimates for the gradient of 
harmonic functions in the unit disk}, Complex Analysis and Operator Theory,
{\bf 7}:4 (2013), 1167--1183.

\bibitem{KHAV}D. Khavinson, \textit{An extremal problem for harmonic functions in the ball}, Canad.  Math. Bull., {\bf 35}:2 (1992), 218--220.

\bibitem{KM} G. Kresin and V. Maz'ya,  \textit{Sharp Real-Part Theorems. A Unified Approach}.
Lect. Notes in Math., \textbf{1903}, Springer-Verlag, Berlin-Heidelberg-New York, 2007.

\bibitem{KM14} G. Kresin and V. Maz'ya, \textit{Sharp real-part theorems in the upper half-plane
and similar estimates for harmonic functions}, J. Math. Sc. (New York), {\bf 179}:1 (2011), 141--163.

\bibitem {KM15} G. Kresin and V. Maz'ya, \textit{Sharp real-part theorems for high order derivatives}, J. Math. Sc. (New York), {\bf 181}:2 (2012), 107--125.

\bibitem {KM17} G. Kresin, \textit{Sharp and maximized real-part estimates for derivatives of analytic functions in the disk}, Rendiconti Lincei - Matematica e Applicazioni, {\bf 24}:1 (2013), 95--110.

\bibitem {KM16} G. Kresin and V. Maz'ya, \textit{Optimal estimates for derivatives of solutions to Laplace,  Lam\'e and Stokes equations},  
J. Math. Sc. (New York), {\bf 196}:3 (2014), 300--321. 

\bibitem{Le} B.Ya. Levin, {\it Lectures on Entire Functions}, Transl. of Math. Monographs,
v. 150, Amer. Math. Soc., Providence, 1996.

\end{thebibliography}
\end{document}